\documentclass[12pt,a4paper,reqno]{amsart}
\allowdisplaybreaks
\usepackage{amsmath}
\usepackage{amsfonts}
\usepackage{mathrsfs}
\usepackage{amssymb,amsthm,amsfonts,amsthm,latexsym,enumerate,url,cases}
\numberwithin{equation}{section}
\usepackage{mathrsfs}
\usepackage{hyperref}
\hypersetup{colorlinks=true,citecolor=blue,linkcolor=blue,urlcolor=blue}
\def\pmod #1{\ ({\rm{mod}}\ #1)}

\theoremstyle{plain}
\newtheorem{theorem}{Theorem}

\newtheorem{problem}{Problem}

\theoremstyle{definition}

\usepackage{etoolbox}
\makeatletter
\patchcmd{\@settitle}{\uppercasenonmath\@title}{}{}{}
\patchcmd{\@setauthors}{\MakeUppercase}{}{}{}
\patchcmd{\section}{\scshape}{}{}{}
\makeatother

\begin{document}

\title
[{A generalization of the Romanoff theorem}]
{A generalization of the Romanoff theorem}

\author
[Y. Ding \quad  and \quad W. Zhai]
{Yuchen Ding \quad {\it and} \quad Wenguang Zhai}

\address{(Yuchen Ding) School of Mathematical Sciences,  Yangzhou University, Yangzhou 225002, People's Republic of China}
\email{ycding@yzu.edu.cn}
\address{(Wenguang Zhai) Department of Mathematics,  China University of Mining and Technology, Beijing 100083, People's Republic of China}
\email{zhaiwg@hotmail.com}


\begin{abstract}

Let  $\mathcal{P}$ be the set of primes and $\mathbb{N}$ the set of positive integers.
Let also $r_1,...,r_t$ be positive real numbers and $R_2(r_1,...,r_t)$ the set of odd integers which can be represented as
$$
p+2^{\lfloor k_1^{r_1}\rfloor}+\cdot\cdot\cdot+2^{\lfloor k_t^{r_t}\rfloor},
$$
where $p\in \mathcal{P}$ and $k_1,...,k_t\in\mathbb{N}$. Recently, Chen and Xu proved that the set $R_2(r_1,...,r_t)$ has positive lower asymptotic density, provided that
$r_1^{-1}+\cdot\cdot\cdot+r_t^{-1}\ge 1$
and at least one of $r_1,...,r_t$ is an integer. Their result reduces to the famous theorem of Romanoff by taking $t=r_1=1.$ In this note, we remove the unnecessary condition: `{\it at least one of $r_1,...,r_t$ is an integer}'.

\end{abstract}
\maketitle

\section{Introoduction}
The aim of this note is the following generalization of the classical theorem named after Romanoff \cite{Romanoff}, which states that the set of odd integers whose elements can be written as the sum of a prime and a power of two has positive lower asymptotic density.

\begin{theorem}\label{thm1}
Let $r_1,...,r_t$ be positive real numbers. Then the set
$$
R_2(r_1,...,r_t)=\left\{\text{odd} ~n:n=p+2^{\lfloor k_1^{r_1}\rfloor}+\cdot\cdot\cdot+2^{\lfloor k_t^{r_t}\rfloor},~p\in \mathcal{P},k_1,...,k_t\in\mathbb{N}\right\}
$$
has positive lower asymptotic density if and only if $r_1^{-1}+\cdot\cdot\cdot+r_t^{-1}\ge 1$, where $\mathcal{P}$ and $\mathbb{N}$ are the sets of primes and positive integers, respectively.
\end{theorem}

This generalization was introduced in a recent article of Chen and Xu \cite{Chen-Xu} but only with an additional condition (which is crucial in their argument): `{\it at least one of $r_1,...,r_t$ is an integer}'.

Before the proof of Theorem \ref{thm1}, We mention the history involving the Romanoff theorem below.
de Polignac \cite{de1} conjectured that every odd number greater than 1 is the sum of a prime and a power of 2. But soon, he \cite{de1} recognized that Euler had already mentioned a counterexample 959 in a letter to Goldbach. Towards the positive direction,
Romanoff \cite{Romanoff} proved that there is a positive proportion of the odd numbers which can be written as the sum of a prime and a power of two. Later, Erd\H os \cite{Er} and van der Corput \cite{va} showed that the odd numbers which can not be represented by the form $p+2^m$ with $p\in\mathcal{P}$ and $m\in \mathbb{N}$ also possess  positive lower asymptotic density. Crocker \cite{Cr} proved that there are infinitely many odd numbers which cannot be represented as $p+2^m+2^n$ with $p\in\mathcal{P}$ and $m,n\in \mathbb{N}$, which was later improved significantly by Chen--Feng--Templier \cite{Ch5} and Pan \cite{Pa1}.

Erd\H os paid much attention to Romanoff's theorem and the proof of it was later simplified by Erd\H os and Tur\'an \cite{Erdos-Turan} and Erd\H os \cite{Erdos}. In \cite{Er}, Erd\H os extended powers of $2$ in the Romanoff theorem to a more general sequence $\{a_k\}$. Precisely, let $a_1<a_2<\cdots$ be an infinite sequence of positive integers satisfying $a_k|a_{k+1}$. Then Erd\H os proved that the necessary and sufficient condition that the sequence $p+a_k$ should have positive density is that
$$
\limsup_{k\rightarrow\infty}\frac{\log a_k}{k}<\infty \quad \text{and} \quad \sum_{d|a_i}\frac{1}{d}<c
$$
for some absolute constant $c$ and any $a_i$.

Following a question of Chen \cite{Chen}, Yang and Chen \cite{Yang-Chen} considered variants of the Romanoff theorem with $c$-condition. A set $\mathcal{B}$ is said to be of $c$-condition if there is some absolute constant $0<c<1$ so that $\mathcal{B}(cx)\gg \mathcal{B}(x)$. Yang and Chen studied the following sumset
$$
\mathcal{S}=\big\{p+b:p\in \mathcal{P},b\in \mathcal{B}\big\}.
$$
They proved that
$$
\frac{x}{\log x}\min\bigg\{\mathcal{B}(x), \frac{\log x}{\log\log x}\bigg\}\ll \mathcal{S}(x)\ll\frac{x}{\log x}\min\big\{\mathcal{B}(x), \log x\big\}.
$$
As an application, Yang and Chen obtained the following estimate
$$
\#\big\{n\le x: n=p+2^{a^2}+2^{b^2},~p\in \mathcal{P},a,b\in \mathbb{N}\big\}\gg \frac{x}{\log\log x}.
$$
Moreover, for any $m$ they constructed a subset $\mathcal{B}$ with $c$-condition so that
$$
\mathcal{B}(x)=\frac{1+o(1)}{m+1}\left(\frac{\log x}{\log\log x}\right)^{m+1}
$$
and $\mathcal{S}(x)\ll x/\log\log x$. These naturally led them to the following two problems.

\begin{problem}\label{p1}
Does there exist a real number $\alpha>0$ and a subset $\mathcal{B}$ of $\mathbb{N}$ with $c$-condition such that $\mathcal{B}(x)\gg x^{\alpha}$ and $\mathcal{S}(x)\ll x/\log\log x$?
\end{problem}

\begin{problem}\label{p2}
Does there exist a positive integer $k$ such that the set of positive integers which can be represented as $p+\sum_{i=1}^{k}2^{m_i^2}$ with $p\in\mathcal{P}$ and $m_i\in \mathbb{N}$ has a positive lower density? If such $k$ exists, what is the minimal value of such $k$?
\end{problem}

In a prior article, the first named author \cite{Ding} completely solved Problem \ref{p2} by showing that $k=2$ is admissible. Thus,
$$
\#\big\{n\le x: n=p+2^{a^2}+2^{b^2},~p\in \mathcal{P},a,b\in \mathbb{N}\big\}\gg x.
$$
In a subsequent article, the first named author \cite{Dingnew} answered Problem \ref{p1} negatively and proved that
for any subset $\mathcal{B}$ of $c$-condition with $\mathcal{B}(x)\gg x^{\alpha}$ we have
$$
\mathcal{S}(x)\gg x/\log\log\log x.
$$
The first named author also pointed out that the above lower bound is sharp via concrete examples. Following the line of Problem \ref{p2}, Chen and Xu \cite{Chen-Xu} obtained Theorem \ref{thm1} with additional condition: `{\it at least one of $r_1,...,r_t$ is an integer}'. And in \cite[Remark 1]{Chen-Xu}, they commented that `{\it We believe that it is unnecessary that at least one of $r_1,...,r_t$ is an integer in our theorem}'.
It is just our purpose to remove this additional condition which forms a clean generalization (Theorem \ref{thm1}) of the Romanoff theorem.

For more variants on the Romanoff theorem, see e.g. Filaseta--Finch--Kozek \cite{FFK}, Pan--Li \cite{Pa3}, Lee \cite{Lee}, Shparlinski--Weingartner \cite{WS},
Elsholtz--Luca--Planitzer \cite{Els}.

\section{Proofs}

Throughout the proof, $\mathcal{S}(x)$ denotes the number of elements of $\mathcal{S}$ not exceeding $x$ and $x$ is supposed to be sufficiently large.
\begin{proof}[Proof of Theorem \ref{thm1}]
One can also easily note that $R_2(r_1,...,r_t)$ is of density zero if $r_1^{-1}+\cdot\cdot\cdot+r_t^{-1}< 1$, making the necessary condition valid. So, we only need to prove the sufficient point.

Suppose first that there exists some $1\le s\le t$ such that $r_s^{-1}\ge 1$, then
$$
\left\{2^{\lfloor k_s^{r_s}\rfloor}:k_s\in\mathbb{N}\right\}=2^{\mathbb{N}}
$$
and hence
$$
\mathcal{P}+2^{\mathbb{N}}\subset R_2(r_1,...,r_t).
$$
In this case, there is nothing to prove for our theorem as it is just the Romanoff theorem.

From now on, we assume that $r_i^{-1}< 1$ for any $1\le i\le t$. So, we would clearly have $t\ge 2$. On observing Chen--Xu's result, we can also assume without loss of generality that none of $r_i$ belongs to integers. The proof begins with the following observation of Chen and Xu \cite[Eq. (4.1)]{Chen-Xu}: There exists an integer $s\ge 2$ such that
$$
\frac{1}{r_1}+\cdot\cdot\cdot+\frac{1}{r_{s-1}}<1\le \frac{1}{r_1}+\cdot\cdot\cdot+\frac{1}{r_{s-1}}+\frac{1}{r_{s}},
$$
from which it follows that there is real number $\lambda\ge 1$ satisfying that
\begin{align}\label{eq18-1}
\frac{1}{r_1}+\cdot\cdot\cdot+\frac{1}{r_{s-1}}+\frac{1}{\lambda r_{s}}=1.
\end{align}
Let
$$
\mathcal{A}=\left\{2^{\lfloor k_1^{r_1}\rfloor}+\cdot\cdot\cdot+2^{\lfloor k_{s-1}^{r_{s-1}}\rfloor}+2^{\left\lfloor \lfloor k_s^\lambda \rfloor^{r_s}\right\rfloor}\right\}.
$$
It suffices to prove that the set $\mathcal{P}+\mathcal{A}$ has positive lower asymptotic density since
$$
\left\{2^{\left\lfloor \lfloor k_s^\lambda \rfloor^{r_s}\right\rfloor}:k_s\in\mathbb{N}\right\} \subseteq\left\{2^{\lfloor k_s^{r_s}\rfloor}:k_s\in\mathbb{N}\right\}.
$$

Denote by $r(n)$ the number of representations of $n=p+a$ with $p\in \mathcal{P}$ and $a\in \mathcal{A}$. We note from Chen and Xu \cite[Pages 77 and 78]{Chen-Xu} that
\begin{align}\label{eq20-1}
\mathcal{A}(x)\asymp \log x.
\end{align}
It is plain that
$$
\sum_{n\le x}r(n)\ge \mathcal{A}(x/2)\mathcal{P}(x/2)\gg \log x(x/\log x) \gg x.
$$
By the Cauchy--Schwarz inequality we have
$$
\bigg(\sum_{n\le x}r(n)\bigg)^2\le \Bigg(\sum_{\substack{n\le x\\ r(n)\ge 1}}1\Bigg) \bigg(\sum_{n\le x}r(n)^2\bigg).
$$
We are going to show that
$$
\sum_{n\le x}r(n)^2\ll x,
$$
from which we would get
$$
\left(\mathcal{P}+\mathcal{A}\right)(x)=\sum_{\substack{n\le x,~r(n)\ge 1}}1\gg x.
$$
Expanding and then exchanging the sum, we obtain
\begin{align*}
\sum_{n\le x}r(n)^2=\sum_{\substack{p_1+a_1=p_2+a_2\le x\\ p_1,p_2\in \mathcal{P},~a_1,a_2\in \mathcal{A}}}1\le \sum_{\substack{a_1,a_2\in \mathcal{A}\\ a_1\le a_2\le x}}\sum_{\substack{p_2\le p_1\le x\\ p_1-p_2=a_2-a_1}}1.
\end{align*}
For the innermost sum in the right--hand side above, we separate it into two cases according to $a_1=a_2$ or not. The contribution of the sum for $a_1=a_2$ can be controlled by
$$
\mathcal{A}(x)\mathcal{P}(x)\ll x
$$
via Eq. (\ref{eq20-1}) and the prime number theorem.
Thus, to complete our proof, it suffices to prove that
\begin{align}\label{eq20-2}
S(x):=\sum_{\substack{a_1,a_2\in \mathcal{A}\\ a_1< a_2\le x}}\sum_{\substack{p_2< p_1\le x\\ p_1-p_2=a_2-a_1}}1\ll x.
\end{align}
For $h\neq0$, let $\pi_2(x,h)$ be the number of prime pairs $p$ and $q$ with $q-p=h$ below $x$. It is well known (see e.g. \cite[Theorem 7.3]{Na}) that
\begin{align}\label{c2}
\pi_2(x,h)\ll\frac{x}{(\log x)^2}\prod_{p|h}\left(1+\frac{1}{p}\right).
\end{align}
Inserting Eq. (\ref{c2}) into Eq. (\ref{eq20-2}) we obtain
\begin{align*}
S(x)\ll \sum_{\substack{a_1,a_2\in \mathcal{A}\\ a_1< a_2\le x}}\frac{x}{(\log x)^2}\prod_{p|a_2-a_1}\left(1+\frac{1}{p}\right).
\end{align*}
We are leading to prove
\begin{align*}
\sum_{\substack{a_1,a_2\in \mathcal{A}\\ a_1< a_2\le x}}\prod_{\substack{p|a_2-a_1\\ p>2}}\left(1+\frac{1}{p}\right)\ll (\log x)^2.
\end{align*}
Following an idea of Elsholtz {\it et al.} \cite{Els}, we note that
$$
\prod_{\substack{p|a_2-a_1\\ p\ge \log x}}\left(1+\frac{1}{p}\right)<\left(1+\frac{1}{\log x}\right)^{\log x/\log 2}\ll 1
$$
since the number of prime factors of $|a_1-a_2|\le x$ is trivially not exceeding $\log x/\log 2$.
Thus, it reduces to show that
\begin{align*}
\sum_{\substack{a_1,a_2\in \mathcal{A}\\ a_1< a_2\le x}}\prod_{\substack{p|a_2-a_1\\ 2<p<\log x}}\left(1+\frac{1}{p}\right)\ll (\log x)^2,
\end{align*}
which is equivalent to
\begin{align}\label{eq27-1}
\sum_{\substack{a_1,a_2\in \mathcal{A}\\ a_1< a_2\le x}}\sum_{\substack{d|a_2-a_1\\ 2\nmid d\\ P^+(d)<\log x}}\frac{\mu^2(d)}{d}\ll (\log x)^2,
\end{align}
where $P^+(d)$ denotes the largest prime factor of $d$ with convention $P^+(1)=1$ and $\mu(d)$ the usual M\"obius function.
On noting the definition of $\mathcal{A}$, it is clear that
\begin{align}\label{eq27-2}
\sum_{\substack{a_1,a_2\in \mathcal{A}\\ a_1< a_2\le x}}\sum_{\substack{d|a_2-a_1\\ 2\nmid d\\ P^+(d)<\log x}}\!\!\!\frac{\mu^2(d)}{d}&\le
\!\!\!\sum_{\substack{k_i,\widetilde{k_i}\le \left(\frac{\log 2x}{\log 2}\right)^{\frac1{r_i}}\\ 1\le i\le s-1}}\sum_{k_s,\widetilde{k_s}\le 2\left(\frac{\log 2x}{\log 2}\right)^{\frac{1}{kr_s}}}\sum_{\substack{d|h_1+\cdot\cdot\cdot+h_{s-1}+h_s-\left(\widetilde{h_1}+\cdot\cdot\cdot+\widetilde{h_{s-1}}+\widetilde{h_s}\right)\\  \widetilde{h_1}+\cdot\cdot\cdot+\widetilde{h_{s-1}}+\widetilde{h_s}<h_1+\cdot\cdot\cdot+h_{s-1}+h_s\le x\\ 2\nmid d,~ P^+(d)<\log x}}\!\!\frac{\mu^2(d)}{d}\nonumber\\
&\ll (\log x)^{2-\frac{1}{r_1}}~\max_{|g|\le x}\sum_{\substack{d\le x,~2\nmid d\\ P^+(d)<\log x}}\frac{\mu^2(d)}{d}\sum_{\substack{k_1\le \left(\frac{\log 2x}{\log 2}\right)^{1/r_1}\\ 2^{\left\lfloor k_1^{r_1}\right\rfloor}\equiv g  \!\pmod{d}}}1,
\end{align}
where the last estimate comes from Eq. (\ref{eq18-1}) and
$$
h_i=2^{\left\lfloor k_i^{r_i}\right\rfloor},  ~~~\widetilde{h_i}=2^{\left\lfloor k_i^{r_i}\right\rfloor}       \quad    (1\le i\le s-1),    \quad      \quad h_s=2^{\left\lfloor \lfloor k_s^\lambda \rfloor^{r_s}\right\rfloor}, ~~~ \widetilde{h_s}=2^{\left\lfloor \left\lfloor \widetilde{k_s}^\lambda \right\rfloor^{r_s}\right\rfloor}.
$$
For some integers $d$, suppose first that the number of $k_1$ satisfying the congruence
\begin{align}\label{eq27-3}
2^{\left\lfloor k_1^{r_1}\right\rfloor}\equiv g  \pmod{d}
\end{align}
is not large than $(\log x)^{1/r_1}/\log\log x$,
then it is admissible for those $d$ since the double sums in right--hand side of Eq. (\ref{eq27-2}) are bounded now by
\begin{align}\label{eq27-bu-1}
\sum_{\substack{d\le x,~2\nmid d\\ P^+(d)<\log x}}\frac{\mu^2(d)}{d}\frac{(\log x)^{1/r_1}}{\log\log x}&\le \frac{(\log x)^{1/r_1}}{\log\log x}\prod_{2<p<\log x}\left(1+\frac{1}{p}\right)\nonumber\\
&=\frac{(\log x)^{1/r_1}}{\log\log x}\exp\left(\sum_{2<p<\log x}\log \left(1+\frac{1}{p}\right)\right)\nonumber\\
&<\frac{(\log x)^{1/r_1}}{\log\log x}\exp\left(\sum_{2<p<\log x}\frac{1}{p}\right)\nonumber\\
&\ll (\log x)^{1/r_1},
\end{align}
where the last estimate follows from the famous theorem of Mertens. We are left over to consider the remaining $d$ for which the number of $k_1$ satisfying Eq. (\ref{eq27-3}) is greater than $(\log x)^{1/r_1}/\log\log x$. In this case, let $k_{1,0}$ be such a solution. Then we have
$$
2^{\left\lfloor k_1^{r_1}\right\rfloor}\equiv 2^{\left\lfloor k_{1,0}^{r_1}\right\rfloor}  \pmod{d}.
$$
So, we would get
\begin{align}\label{eq28-1}
\left\lfloor k_1^{r_1}\right\rfloor\equiv \ell \pmod{e_2(d)}
\end{align}
for some $0\le\ell\le e_2(d)-1$, where $e_2(d)$ is the least positive integer $m$ so that $d|2^m-1$. Furthermore, in case that (\ref{eq28-1}) has more than $(\log x)^{1/r_1}/\log\log x$ solutions of $k_1$, we claim that
\begin{align}\label{key}
e_2(d) \ll_{r_1} (\log x)^{1-\frac{1}{r_1}}\log\log x,
\end{align}
which would be one of the key points in the following arguments.
In fact, suppose that $L$ is the minimal distance between two consecutive solutions $k_1'$ and $k_1'+L$ of the congruence equation (\ref{eq28-1}) with $k_1'>(\log x)^{\frac{1}{2r_1}}$, then
the number of solutions $k_1$ to (\ref{eq28-1}) is not more than
\begin{align*}
(\log x)^{\frac{1}{2r_1}}+\frac{(\log x)^{1/r_1}}{L}+1.
\end{align*}
Hence, by our assumption on the number of the solutions, we have
\begin{align*}
(\log x)^{\frac{1}{2r_1}}+\frac{(\log x)^{1/r_1}}{L}+1\ge \frac{(\log x)^{1/r_1}}{\log\log x},
\end{align*}
which clearly means that
$
L\le 2\log\log x.
$
By the definition of $L$, we also have
\begin{align*}
\lfloor(k_1'+L)^{r_1}\rfloor-\lfloor(k_1')^{r_1}\rfloor\ge e_2(d),
\end{align*}
from which it follows that
\begin{align*}
e_2(d)\le (k_1'+L)^{r_1}-(k_1')^{r_1}+1\ll_{r_1} (k_1')^{r_1-1}L\ll (\log x)^{1-1/r_1}\log\log x.
\end{align*}
We now conclude from (\ref{eq27-2}), (\ref{eq28-1}), (\ref{key}) and the above proofs that we only need to show
\begin{align}\label{additional-5}
\sum_{\substack{d\le x,~2\nmid d\\ P^+(d)<\log x\\e_2(d) \le (\log x)^{1-1/r_1}\log\log x}}\frac{\mu^2(d)}{d}\sum_{\substack{k_1\le \left(\frac{\log 2x}{\log 2}\right)^{1/r_1}\\ \lfloor k_1^{r_1}\rfloor\equiv \ell \pmod{e_2(d)}}}1\ll (\log x)^{\frac{1}{r_1}}.
\end{align}

For any $k_1\le (\log x)^{1/r_1}$ satisfies (\ref{eq28-1}), we clearly have
$$
e_2(d)q+\ell\leq k_1^{r_1}<e_2(d)q+\ell+1
$$
for some nonnegative integer $q$, which is equivalent to
\begin{align}\label{DZ-new-1}
\frac{k_1^{r_1}-\ell-1}{e_2(d)}<q\leq \frac{k_1^{r_1}-\ell}{e_2(d)}.
\end{align}
It is plain that an integer $q$ satisfies (\ref{DZ-new-1}) if and only if
$$
\left\lfloor\frac{k_1^{r_1}-\ell}{e_2(d)}\right\rfloor-\left\lfloor\frac{k_1^{r_1}-\ell-1}{e_2(d)}\right\rfloor=1,
$$
from which it follows that
\begin{align}\label{DZ-new-2}
\sum_{\substack{k_1\le \left(\frac{\log 2x}{\log 2}\right)^{1/r_1}\\ \lfloor k_1^{r_1}\rfloor\equiv \ell \pmod{e_2(d)}}}1&\le \sum_{\substack{k_1\le \left(\frac{\log 2x}{\log 2}\right)^{1/r_1}}}\bigg(\left\lfloor\frac{k_1^{r_1}-\ell}{e_2(d)}\right\rfloor-\left\lfloor\frac{k_1^{r_1}-\ell-1}{e_2(d)}\right\rfloor\bigg)\nonumber\\
&=\sum_{\substack{k_1\le \left(\frac{\log 2x}{\log 2}\right)^{1/r_1}}}\bigg(\frac{1}{e_2(d)}-\psi\left(\frac{k_1^{r_1}-\ell}{e_2(d)}\right)-\psi\left(\frac{k_1^{r_1}-\ell-1}{e_2(d)}\right)\bigg)\nonumber\\
&\ll_{r_1} \frac{(\log x)^{1/r_1}}{e_2(d)}+O\bigg(\max_{|\theta|\le 1}\bigg|\sum_{\substack{k_1\le \left(\frac{\log 2x}{\log 2}\right)^{1/r_1}}}\psi\left(\frac{k_1^{r_1}}{e_2(d)}+\theta\right)\bigg|\bigg),
\end{align}
where $\psi(t)=\{t\}-1/2$ is the sawtooth function.
Now, we write
\begin{align}\label{DZ-new-3}
\sum_{\substack{k_1\le \left(\frac{\log 2x}{\log 2}\right)^{1/r_1}}}\psi\left(\frac{k_1^{r_1}}{e_2(d)}+\theta\right)=S_1+S_2,
\end{align}
where
\begin{align*}
S_1=   \sum_{k\leq e_2(d)^{1/r_1}}\psi\left(\frac{k_1^{r_1}}{e_2(d)}+\theta\right) \quad \text{and} \quad
 S_2=   \sum_{e_2(d)^{1/r_1}<k\leq \left(\frac{\log 2x}{\log 2}\right)^{1/r_1}}\psi\left(\frac{k_1^{r_1}}{e_2(d)}+\theta\right).
\end{align*}
Trivially, we have
\begin{align}\label{DZ-new-4}
S_1   \ll e_2(d)^{1/r_1} .
\end{align}
We are now in a position to use the following van der Corput type bounds \cite[Lemma 4.3]{Graham} for a suitable estimate of $S_2$.

{\bf Lemma 1.} {\it For any real number $t,$ let $\psi(t)=\{t\}-1/2$. Suppose $ K\geq 3$ is a real number, $Y>0.$  Suppose $\alpha>0~ (\alpha\notin \mathbb{Z})$ is a fixed real number. Then for any exponent pair $(\kappa, \lambda) $ we have
\begin{eqnarray*}
\sum_{K\le k<2 K}\psi\left( Y k^{\alpha}+\theta \right) \ll
\frac{K^{1-\alpha}}{Y}+
Y^{\frac{\kappa}{1+\kappa}}K^{\frac{\lambda+\kappa\alpha}{1+\kappa}}.
\end{eqnarray*}
}

For $ e_2(d)^{1/r_1}\le K\le \left(\frac{\log 2x}{\log 2}\right)^{1/r_1}$, by {\bf Lemma 1} with $Y=\frac{1}{e_2(d)} $ we have
\begin{align*}
  \sum_{K\le k_1<2K}\psi\left(\frac{k_1^{r_1}}{e_2(d)}+\theta\right)&= \sum_{K\le k<2K}\psi\left(Yk_1^{r_1}\right)\\
  &
  \ll   \frac{1}{ Y K^{r_1-1}}+ Y^{\frac{\kappa}{1+\kappa}} K^{\frac{r_1\kappa +\lambda}{1+\kappa}}\\
  &
  \ll \frac{e_2(d)}{ K^{r_1-1}}+e_2(d)^{-\frac{\kappa}{1+\kappa}} K^{\frac{r_1\kappa +\lambda}{1+\kappa}}
  \end{align*}
since $Y= 1/e_2(d)$, which combining a splitting argument yields
\begin{align}\label{DZ-new-5}
 S_2= \sum_{e_2(d)^{1/r_1}<k_1\leq \left(\frac{\log 2x}{\log 2}\right)^{1/r_1}}\psi\left(\frac{k^{r_1}}{e_2(d)}+\theta\right)
 \ll      (\log x)^{\frac{1}{r_1}\cdot\frac{r_1\kappa +\lambda}{1+\kappa}}
   e_2(d)^{-\frac{\kappa}{1+\kappa}}+ e_2(d)^{1/r_1}.
\end{align}
Taking (\ref{DZ-new-4}) and (\ref{DZ-new-5}) into (\ref{DZ-new-3}), we have
\begin{align*}
\sum_{\substack{k_1\le \left(\frac{\log 2x}{\log 2}\right)^{1/r_1}\\ \lfloor k_1^{r_1}\rfloor\equiv \ell \pmod{e_2(d)}}}1\ll \frac{(\log x)^{1/r_1}}{e_2(d)}+(\log x)^{\frac{1}{r_1}\cdot\frac{r_1\kappa +\lambda}{1+\kappa}}
   e_2(d)^{-\frac{\kappa}{1+\kappa}}+ e_2(d)^{1/r_1}.
\end{align*}
Therefore, we have
\begin{align}\label{DZ-new-6}
\sum_{\substack{d\le x,~2\nmid d\\ P^+(d)<\log x\\e_2(d) \le (\log x)^{1-1/r_1}\log\log x}}\frac{\mu^2(d)}{d}\sum_{\substack{k_1\le \left(\frac{\log 2x}{\log 2}\right)^{1/r_1}\\ \lfloor k_1^{r_1}\rfloor\equiv \ell \pmod{e_2(d)}}}1\ll_{r_1} \mathcal{W}_1+\mathcal{W}_2+\mathcal{W}_3,
\end{align}
where
$$
\mathcal{W}_1=\sum_{\substack{d\le x,~2\nmid d\\ P^+(d)<\log x\\e_2(d) \le (\log x)^{1-1/r_1}\log\log x}}\frac{\mu^2(d)}{d}\frac{(\log x)^{1/r_1}}{e_2(d)},
$$
$$
\mathcal{W}_2=\sum_{\substack{d\le x,~2\nmid d\\ P^+(d)<\log x\\e_2(d) \le (\log x)^{1-1/r_1}\log\log x}}\frac{\mu^2(d)}{d}(\log x)^{\frac{1}{r_1}\cdot\frac{r_1\kappa +\lambda}{1+\kappa}}
   e_2(d)^{-\frac{\kappa}{1+\kappa}}
$$
and
$$
\mathcal{W}_3=\sum_{\substack{d\le x,~2\nmid d\\ P^+(d)<\log x\\e_2(d) \le (\log x)^{1-1/r_1}\log\log x}}\frac{\mu^2(d)}{d}e_2(d)^{1/r_1}.
$$
By a result of Erd\H{o}s and Tur\'{a}n \cite{ET}, we have
\begin{equation}\label{e17}
\sum_{2\nmid d}\frac{1}{d(e_2(d))^{\varepsilon}}<\infty
\end{equation}
for any $\varepsilon>0$. Using (\ref{e17}), we know that
$$
\mathcal{W}_1\le (\log x)^{1/r_1}\sum_{\substack{2\nmid d}}\frac{1}{de_2(d)}\ll(\log x)^{1/r_1}
$$
and together with the estimate given in (\ref{eq27-bu-1}), we get
\begin{align*}
\mathcal{W}_3&\ll \Big((\log x)^{1-1/r_1}\log\log x\Big)^{1/r_1}\sum_{\substack{d\le x,~2\nmid d\\ P^+(d)<\log x\\e_2(d) \le (\log x)^{1-1/r_1}\log\log x}}\frac{\mu^2(d)}{d}\\
&\ll (\log x)^{\frac{1}{r_1}-\frac{1}{r_1^2}}(\log\log x)^{\frac{1}{r_1}}\sum_{\substack{d\le x,~2\nmid d\\ P^+(d)<\log x }}\frac{\mu^2(d)}{d}\\
&\ll(\log x)^{\frac{1}{r_1}-\frac{1}{r_1^2}}(\log\log x)^{\frac{1}{r_1}+1}\\
&\ll (\log x)^{\frac{1}{r_1}}.
\end{align*}
On noting (\ref{DZ-new-6}), to establish the estimates of (\ref{additional-5}) it remains to show $\mathcal{W}_2\ll (\log x)^{\frac{1}{r_1}}$. Again, by (\ref{e17}) and the estimates in (\ref{eq27-bu-1}), it suffices to choose a suitable exponent pair $(\kappa,\lambda)$ such that
\begin{align}\label{DZ-new_10}
\frac{r_1\kappa +\lambda}{1+\kappa}<1.
\end{align}
For any integer $q\geq 1$, we know that
$$
\left(\frac1{4Q-2},1-\frac{q+1}{4Q-2}\right)
$$
is an exponent pair, thanks to the display of \cite[Page 60]{Graham},
where $Q=2^q$. We now take $q=\lfloor r_1\rfloor+1$ and then
$$
r_1\kappa +\lambda=\frac{r_1}{4Q-2}+1-\frac{q+1}{4Q-2}<1,
$$
from which $(\ref{DZ-new_10})$ clearly follows.

\end{proof}

\section*{Acknowledgments}
We thank Rui-Jing Wang for his interests and pointing out some inaccuracies in the arguments of the former version.

Yuchen Ding is supported by National Natural Science Foundation of China  (Grant No. 12201544), Natural Science Foundation of Jiangsu Province, China (Grant No. BK20210784), China Postdoctoral Science Foundation (Grant No. 2022M710121).

Wenguang Zhai is supported by National Natural Science Foundation of China  (Grant Nos. 12471009, 12301006) and    by  Beijing Natural Science Foundation (Grant No. 1242003).


\begin{thebibliography}{KMP}
\bibitem{Chen} Y.--G. Chen, {\it Romanoff theorem in a sparse set,} Sci. China Math. {\bf 53} (2010), 2195--2202.

\bibitem{Ch5} Y.--G. Chen, R. Feng and N. Templier, {\it Fermat numbers and integers of the form $a^k + a^l + p^\alpha$,} Acta Arith. {\bf135} (2008), 51--61.

\bibitem{Chen-Xu} Y.--G. Chen and J.--Z. Xu, {\it On integers of the form $p+2^{k_1^{r_1}}+\cdot\cdot\cdot+2^{k_t^{r_t}}$,} J. Number Theory {\bf 258} (2024), 66--93.

\bibitem{va} J. G. van der Corput, {\it On de Polignac's conjecture,} Simon Stevin {\bf27} (1950), 99--105.

\bibitem{Cr} R. Crocker, {\it On the sum of a prime and two powers of two,} Pacific J. Math. {\bf 36} (1971), 103--107.

\bibitem{Ding} Y. Ding, {\it On a problem of Romanoff type,} Acta Arith. {\bf 205} (2022), 53--62.

\bibitem{Dingnew} Y. Ding, {\it Solutions to a Romanoff type theorem,} to appear in Acta Arith. arXiv:2403.12874

\bibitem{Els} C. Elsholtz, F. Luca and S. Planitzer, {\it Romanov type problems,} Ramanujan J. {\bf47} (2018), 267--289.

\bibitem{Erdos} P. Erd\H{o}s, {\it On some problems of Bellman and a theorem of Romanoff,} J. Chinese Math. Soc. {\bf 2} (1950), 113--123.

\bibitem{Er} P. Erd\H{o}s, {\it On the integers of the form $2^k+p$ and some related problems,} Summa Brasil. Math. {\bf 2} (1950), 113--123.

\bibitem{Erdos-Turan} P. Erd\H{o}s and P. Tur\'an, {\it Ein zahlentheoretischer,} Satz. Izv. Inst. Math. Mech. Tomsk State Univ. {\bf 1} (1935), 101--103.

\bibitem{ET} P. Erd\H{o}s and P. Tur\'{a}n, {\it \"{U}ber die Vereinfachung eines Landauschen Satzes,} Mitt. Forsch.-Inst. Math. Mech. Univ. Tomsk, {\bf1} (1935), 144--147.

\bibitem{FFK} M. Filaseta, C. Finch and M. Kozek, {\it On powers associated with Sierp\'{\i}nski numbers, Riesel numbers and Polignac's conjecture,} J. Number Theory {\bf128} (2008), 1916--1940.

\bibitem{Graham} S. W. Graham and G. Kolesnik, {\it Van der Corput's Method of Exponential Sums,} Cambridge University Press, 1991.

\bibitem{Lee} K.S.E. Lee, {\it On the sum of a prime and a Fibonacci number,} Int. J. Number Theory {\bf 6} (2010), 1669--1676.


\bibitem{Na} M. B. Nathanson, {\it Additive Number Theory: The Classical Bases}, Springer, 1996.

\bibitem{Pa1} H. Pan, {\it On the integers not of the form $p + 2^a + 2^b$,} Acta Arith. {\bf 148} (2011), 55--61.

\bibitem{Pa3} H. Pan and  H. Li, {\it The Romanoff theorem revisited,} Acta Arith. {\bf 135} (2008), 137--142.

\bibitem{de1} A. de Polignac, {\it Six propositions arithmologiques d\'{e}duites du crible d'Eratosth\'{e}ne,} Nouv. Ann. Math. {\bf8} (1849), 423--429.

\bibitem{de2} A. de Polignac, {\it Recherches nouvelles sur ies nombres primiers,} C. R. Acad. Sci. Paris {\bf 29} (1849), 738--739.

\bibitem{Romanoff} N.P. Romanoff, {\it \"{U}ber einige S\"{a}tze der additiven Zahlentheorie,} Math. Ann. {\bf 109} (1934), 668--678.

\bibitem{WS}  I. Shparlinski and A. Weingartner, {\it An explicit polynomial analogue of Romanoff's theorem,} Finite Fields Appl. {\bf44} (2017), 22--33.

\bibitem{Yang-Chen} Q.--H. Yang and Y.--G. Chen, {\it On the integers of the form $p+b$,} Taiwanese J. Math. {\bf 18} (2014), 1623--1631.

\end{thebibliography}
\end{document}